\documentclass[12pt]{amsart}

\usepackage{fullpage}
\usepackage{amsmath}
\usepackage{amssymb}
\usepackage{amsthm}
\usepackage{amstext}
\usepackage{appendix}
\usepackage{perpage}
\usepackage[x11names]{xcolor}
\usepackage{tikz-cd}
\usepackage{ytableau}\ytableausetup{centertableaux}
\usepackage[colorlinks=true, allcolors=blue]{hyperref}

\newtheorem{thm}{Theorem}[section]
\newtheorem{lemm}[thm]{Lemma}
\newtheorem{coro}[thm]{Corollary}
\newtheorem{prop}[thm]{Proposition}

\theoremstyle{definition}

\newtheorem{defi}[thm]{Definition}
\newtheorem{remark}[thm]{Remark}

\begin{document}

\title{The algebraisation of higher level Deligne--Lusztig representations II: odd levels}

\author{Zhe Chen \and Alexander Stasinski}

\address{Department of Mathematics, Shantou University, Shantou, 515063, China}
\email{zhechencz@gmail.com}

\address{Department of Mathematical Sciences, Durham University, Durham, DH1 3LE, UK}
\email{alexander.stasinski@durham.ac.uk}

\begin{abstract}
In this paper we study higher level Deligne--Lusztig representations of reductive groups over discrete valuation rings, with finite residue field $\mathbb{F}_q$. In previous work we proved that, at even levels, these geometrically constructed representations are isomorphic to certain algebraically constructed representations (referred to as the algebraisation theorem at even levels).  In this paper we work with an arbitrary level $>1$. 

Our main result is

\begin{itemize}
\item[(1)] the algebraisation theorem at all levels $>1$ (with the sign being explicitly determined for $q\geq7$). As consequences, we obtain

\item[(2)] the regular semisimplicity of orbits of generic higher level Deligne--Lusztig representations, and the dimension formula; in the course of the proof, we give

\item[(3)] an induction formula of higher level Deligne--Lusztig representations, and a new proof of the character formula at regular semisimple elements.
\end{itemize}
\end{abstract}

\maketitle

\tableofcontents

\section{Introduction}

Let $\mathbb{G}$ a connected reductive group scheme, defined over   a complete discrete valuation ring $\mathcal{O}$ with a finite residue field $\mathbb{F}_q$. We will consider two approaches to construct smooth representations of $\mathbb{G}(\mathcal{O})$: An algebraic one based on Clifford theory, and a geometric one based on \'etale cohomology. We shall be concerned with a comparison between them.

\vspace{2mm} Let $\pi$ be a fixed uniformiser of $\mathcal{O}$, and write $\mathcal{O}_r:=\mathcal{O}/\pi^r$, $r\in\mathbb{Z}_{>0}$. By  basic properties of profinite topology, the study of smooth representations of $\mathbb{G}(\mathcal{O})$ is equivalent to the study of representations of the finite groups $\mathbb{G}(\mathcal{O}_r)$, for all $r\in\mathbb{Z}_{>0}$. In the below, the word ``level'' refers to the positive integer $r$.

\vspace{2mm} For $r=1$, the groups $\mathbb{G}(\mathcal{O}_1)=\mathbb{G}(\mathbb{F}_q)$ are the so-called finite groups of Lie type, and a quite satisfactory theory of their representations was established in \cite{DL1976} (and extensive further works, mainly due to Lusztig), building on $\ell$-adic cohomology, known as Deligne--Lusztig theory. In \cite[Section~4]{Lusztig1979SomeRemarks} Lusztig proposed a generalisation of this construction for an arbitrary $r\in\mathbb{Z}_{>0}$, and this was later proved  in \cite{Lusztig2004RepsFinRings} when $\mathrm{char}(\mathcal{O})>0$ and in \cite{Sta2009Unramified} in general.

\vspace{2mm} When $r=1$, Deligne--Lusztig theory is the only known way to construct all irreducible representations for a general $\mathbb{G}$. However, when $r\geq 2$, besides the geometric method, there is also an algebraic approach (based on Clifford theory) to construct certain representations of $\mathbb{G}(\mathcal{O}_r)$, going back to the works of Shintani \cite{Shintani1968sqr_int_irrep_lin_gr} (for $\mathrm{GL}_n$), G\'erardin \cite{Gerardin1975SeriesDiscretes}, and, independently, Hill \cite{Hill_1995_Regular} (for $\mathrm{GL}_n$); in this aspect, G\'erardin's construction works for a general $\mathbb{G}$ whose derived subgroup is simply-connected, and shares a same set of parametres with Lusztig's construction, the set of characters of the maximal tori. While the generalised Deligne--Lusztig theory  provides a unified treatment for all $r\geq 1$, there is a quite natural problem,  originally raised by Lusztig in \cite[Introduction]{Lusztig2004RepsFinRings}: For $r\geq2$, are the geometrically constructed representations isomorphic to G\'erardin's representations, along with the tori characters under suitable generic conditions?

\vspace{2mm} When $r>1$ is even, G\'erardin's representations admit a transparent cohomological realisation, which allows us to compare them with the higher Deligne--Lusztig representations using $\ell$-adic cohomology, via an inner product formula; this leads to a positive answer to the above question for even $r$ in \cite{ChenStasinski_2016_algebraisation} (see the formulation in Theorem~\ref{thm:main_even}).

\vspace{2mm} When $r>1$ is odd, the situation becomes much more subtle. Indeed, in this case there are two non-trivial extensions of representations (involving the so-called Heisenberg lifts) on the algebraic side; in the even level case, one of the extensions disappears, and the other becomes a trivial inflation. Furthermore, for odd $r$ there is a non-trivial sign function, which is eliminated in the even level case. As such, for odd $r$ there is no more a direct cohomological realisation on the algebraic side. To tackle the difficulty we need to employ new ingredients, and our basic strategy consists of three steps:
\begin{itemize}
\item The first is to pass from $\mathbb{G}(\mathcal{O}_r)$ to a simpler subgroup $(TG^{l'})^F< \mathbb{G}(\mathcal{O}_r)$ (see the paragraphs above Definition~\ref{defi: generic conditions}), by two inner product formulae (see Proposition~\ref{prop:middle level step} and Lemma~\ref{lemm: statement (1)}), one of which actually generalises the key step in the even level case;
\item the second step is to construct and characterise a family of representations of the above subgroup by algebraic methods, in a way more conceptual than G\'erardin's ealier approach (see Lemma~\ref{lemm:existence of Heisenberg lift} and Proposition~\ref{prop:existence of ext of rho'});
\item in the final step, the major task is to determine the sign of a virtual representation of $(TG^{l'})^F$, for which the techniques include Brou\'e's bimodule formula, an inductive argument based on dimensions of centralisers of semisimple elements, and a reduction to simply-connected covers (see Lemma~\ref{lemm:full alg of H(X_0)}).
\end{itemize}

\vspace{2mm} Towards the first step in the above strategy is an induction formula of higher level Deligne--Lusztig representations $R_{T,U}^{\theta}$, namely, for an arbitrary $r>1$, we prove that (Theorem~\ref{thm:induction theorem}): 
$$R_{T,U}^{\theta}\cong\mathrm{Ind}_{(TG^{l'})^F}^{G^F} H_c^*(X_0)_{\theta},$$
where $H_c^*(X_0)_{\theta}$ is, up to a sign, an irreducible representation of $(TG^{l'})^F$.

\vspace{2mm} To deduce the above induction formula we need to compute certain inner products, which in turn imply the following result: Assuming the existence of a non-degenerate bilinear form on the Lie algebra (see \hyperlink{condition (A)}{{(A)}}), one can define the notion of orbit $\Omega(\sigma)$ in Lie algebra (Definition~\ref{defi: orbit rep}) for an irreducible representation $\sigma$ of $\mathbb{G}(\mathcal{O}_r)$. In Theorem~\ref{thm:orbit type} we prove that 
\begin{equation*}
\Omega(\pm R_{T,U}^{\theta})\ \textrm{is regular and semisimple},
\end{equation*} 
for every generic $\theta$. This gives an affirmative answer to the problem posed in \cite[Introduction]{Stasinski_Stevens_2016_regularRep}. Then for $\mathbb{G}=\mathrm{GL}_n$, in Proposition~\ref{coro:dim formula} we solve the regular semisimple case of a claim made in \cite[Section~2. Remarks]{Hill_1993_Jordan}, which asserts that the notion of orbits can be described by an analogue of Deligne--Lusztig characters. Meanwhile, as an immediate application of the above induction formula, we find a new proof of the character formula in \cite[Theorem~4.1]{Chen_2018_GreenFunction}:  
$$\mathrm{Tr}(s,R_{T,U}^{\theta})=\sum_{w\in W(T)^F}(^w\theta)(s),$$
where $s\in T_1^F$ is regular semisimple and $\theta$ is strongly generic (Corollary~\ref{coro:character formula at rss}).

\vspace{2mm} In Section~\ref{sec:alg construction}, we prove our main result (see Theorem~\ref{thm:full alg}), the algebraisation theorem for every $r>1$:
$$R_{T,U}^{\theta}\cong e_{\theta}(1)^r\cdot \mathrm{Ind}_{(TG^{l'})^F}^{G^F} \hat{\rho}_{\theta},$$
where $\theta$ is strongly generic, 
$\hat{\rho}_{\theta}$ is an irreducible  representation of $(TG^{l'})^F$ constructed by algebraic methods, and $e_{\theta}(1)=\pm 1$; furthermore, the precise value of $e_{\theta}(1)$ is determined provided $q\geq7$. This result gives a positive solution to the problem raised in \cite[Introduction]{Lusztig2004RepsFinRings}. It immediately implies the dimension formula
$$\dim R_{T,U}^{\theta}=\pm {|G_1^F/T_1^F|_{p'}}\cdot q^{(r-1)\cdot\#\Phi^+},$$
where $\theta$ is strongly generic and $\Phi^+$ denotes the set of positive roots (Corollary~\ref{coro:dim formula for general case}). Prior to this result, the dimension of $R_{T,U}^{\theta}$ was not known for odd $r>1$. We shall remark that, the proof of the  dimension formula for the classical $r=1$ case uses special properties of the Steinberg character, which are not available for $r>1$. We shall also note that, while the algebraisation theorem is for $r>1$, the dimension formula and the sign that appears  are compatible with the classical $r=1$ case. In the special case that $\mathbb{G}=\mathrm{GL}_n$, the sign $e_{\theta}(1)$ is determined for all $q$ (see Remark~\ref{remark:q<7 for GL_n}).

\vspace{2mm} One of the main points of the above algebraisation is to make the higher Deligne--Lusztig representations more ``explicit''; it is only through this isomorphism we can determine the basic invariants like orbits and dimensions. On the other hand, the higher Deligne--Lusztig construction has an important feature that it implies a natural Frobenius action  on the representation space $R_{T,U}^{\theta}$, through cohomology groups; it is unclear how to construct this action from the algebraic side.

\vspace{2mm} A brief summary of each section: In Section~\ref{sec:prelim} we explain several basic notations and concepts; in particular we show the equivalences between the generic conditions for $\mathrm{GL}_n$ (see Proposition~\ref{prop: regular in GL_n}). In Section~\ref{sec:main strategy} and Section~\ref{sec:alg construction} we present our main results mentioned above, which include the induction formula, the orbit type determination, the full algebraisation etc. The final Section~\ref{sec: tech1} is devoted to the proof of a central ingredient, Lemma~\ref{lemm: statement (1)}, which guarantees the irreducibility of $H_c^*(X_0)_{\theta}$ for regular $\theta$, up to sign.

\vspace{2mm} \noindent {\bf Acknowledgement.} During the preparation of this work, ZC is partially supported  by NSFC No.~12001351 and Natural Science Foundation of Guangdong No.~2023A1515010561.

\section{Preliminaries}\label{sec:prelim}

In this section we briefly recall the notation and the constructions needed, mainly following our first paper \cite{ChenStasinski_2016_algebraisation}.

\vspace{2mm} Let $\mathcal{O}^{\mathrm{ur}}$ be the ring of integers of the maximal unramified extension of $\mathrm{Frac}(\mathcal{O})$, and let $\mathcal{O}^{\mathrm{ur}}_r:=\mathcal{O}^{\mathrm{ur}}/\pi^r$. Given an affine smooth group scheme $\mathbf{G}$ over $\mathcal{O}^{\mathrm{ur}}_r$, by the Greenberg functor (see \cite{Greenberg19611},\cite{Greenberg19632},\cite{Sta2009Unramified}) there is an associated affine smooth algebraic group $G_r$ over $\overline{\mathbb{F}}_q$ such that 
$$G_r(\overline{\mathbb{F}}_q)\cong \mathbf{G}(\mathcal{O}^{\mathrm{ur}}_r)$$ 
as abstract groups. When there is no confusion, we often drop the subscript $r$ and use the notation $G=G_r$. 

\vspace{2mm} From now on, let $\mathbb{G}$ be a connected reductive group scheme over $\mathcal{O}_r$ in the sense of \cite[XIX~2.7]{SGA3}, and let $\mathbf{G}$ be its base change to $\mathcal{O}^{\mathrm{ur}}_r$. Then there is an associated geometric Frobenius endomorphism $F$ of $G$ such that 
$$G^F\cong \mathbb{G}(\mathcal{O}_r)$$ 
as abstract groups; this allows us to use the geometry of $G$ to study the representations of $\mathbb{G}(\mathcal{O}_r)$. If $i\leq r$ is a positive integer, then the reduction map modulo $\pi^{i}$ induces a surjective algebraic  group morphism 
$$\rho_{r,i}\colon G\longrightarrow G_i,$$
and we denote the kernel of $\rho_{r,i}$ by $G_r^i=G^i$. Similar notation applies to the closed subgroups of $G$.

\vspace{2mm} Let $\mathbf{B}$ be a Borel subgroup of $\mathbf{G}$ and let $\mathbf{B}=\mathbf{T}\ltimes\mathbf{U}$ be a Levi decomposition (with $\mathbf{T}$ being a maximal torus and $\mathbf{U}$ being the unipotent radical of $\mathbf{B}$). Then we have the associated algebraic groups $B,T,U$ (as well as the opposite groups $B^-,U^-$). Throughout this paper we assume that $T$ is $F$-stable (i.e.,\ $FT=T$). Let $L\colon g\mapsto g^{-1}F(g)$ be the associated Lang isogeny. Note that $G^F$ acts on the variety $L^{-1}(FU)$ by left translation, and $T^F$ acts on $L^{-1}(FU)$ by right translation. The Deligne--Lusztig representations at level $r$ are the virtual representations
$$R_{T,U}^{\theta}:=\sum_i(-1)^iH_c^i(L^{-1}(FU),\overline{\mathbb{Q}}_{\ell})_{\theta}$$
of $G^F$, where $H_c^i(-,\overline{\mathbb{Q}}_{\ell})$ denotes the compactly supported $\ell$-adic cohomology (here $\ell$ is an arbitrary fixed prime not equal to $p:=\mathrm{char}(\mathbb{F}_q)$), and the subscript ${(-)}_{\theta}$ means that we are taking the $\theta$-isotypical part for the irreducible characters $\theta\in \mathrm{Irr}({T^F})$.

\vspace{2mm} For a given $r>1$, consider the two integers: $l=\lceil\frac{r}{2}\rceil$ and $l'=\lfloor\frac{r}{2}\rfloor$ (in particular, $l=l'$ if $r$ is even). Relevant to $l$ and $l'$, there are some $F$-stable closed subgroups of $G$ that are important to us:
$$G^l, G^{l'}, U^{\pm}:=U^l(U^-)^l, TG^l=TU^{\pm}, TG^{l'}.$$
Note that $G^{l}$  is commutative  and its elements commute with the elements in $G^{l'}$ (since $[G^i,G^j]\subseteq G^{i+j}$).

\vspace{2mm} Let $\Phi$ denote the set of roots determined by $\mathbf{T}$, and let $\Phi^+$ (resp.\ $\Phi^-$) denote the subset of positive (resp.\ negative) roots with respect to $\mathbf{B}$. Then for any $\alpha\in \Phi^+$ (resp.\ $\in \Phi^-$) we have a closed subgroup $U_{\alpha}\subseteq U$ (resp.\ $\subseteq U^-$). Let $T^{\alpha}\subseteq T$ be the image of the coroot $\check{\alpha}$, and let $\mathcal{T}^{\alpha}:=(T^{\alpha})^{r-1}$.

\begin{defi}\label{defi: generic conditions}
Let $r>1$. Here we list various conditions on $\theta\in\mathrm{Irr}({T^F})$:
\begin{itemize}
\item[(1)] $\theta$ is called \emph{in general position}, if no non-trivial element in $(N_G(T)/T)^F$ fixed $\theta$.
\item[(2)] $\theta$ is called \emph{regular}, if it is non-trivial on $N^{F^a}_F((\mathcal{T}^{\alpha})^{F^a})$ for every $\alpha\in\Phi$, where $a\in\mathbb{Z}_{>0}$ is such that $F^a\mathcal{T}^{\alpha}=\mathcal{T}^{\alpha}$ for $\forall \alpha\in\Phi$, and $N^{F^a}_F(t):=t\cdot F(t)\cdots F^{a-1}(t)$ is the norm map on $T^{F^a}$. (This condition is independent of the choice of $a$ \cite[2.8]{Sta2009Unramified}.) 
\item[(3)] $\theta$ is called \emph{generic}, if it is regular and in general position.
\item[(4)] $\theta$ is called \emph{strongly generic}, if it is generic and satisfies the stabiliser condition $\mathrm{Stab}_{G^F}(\widetilde{\theta}|_{(G^l)^F})=(TG^{l'})^F$, where $\widetilde{\theta}$ is the inflation of $\theta$ to $(TG^l)^F$ via $TG^l/U^{\pm}\cong T$.
\end{itemize} 
\end{defi}

We remark that, in the above definition being generic is very close to being strongly generic, and they are natural analogues of (strongly) regular semisimple elements in the Lie algebra. In the next proposition, for $\mathrm{GL}_n$, we prove that these two notions are actually equivalent, and that most $\theta$'s are strongly generic when $q$ is large; for a general $\mathbb{G}$ with $r=2,3$, this also holds true assuming the existence of a certain non-degenerate bilinear form on the Lie algebra  (see Remark~\ref{remark:strongly generic at r=2,3}).

\begin{prop}\label{prop: regular in GL_n}
Let $\mathbb{G}=\mathrm{GL}_n$. Then for $\theta\in\mathrm{Irr}({T^F})$,
$$\theta\ \textrm{is regular}
\Longleftrightarrow \theta\ \textrm{is generic}
\Longleftrightarrow \theta\ \textrm{is strongly generic},$$
and we have
\begin{equation*}
\lim_{m\rightarrow+\infty}\frac{\#\{\textrm{strongly generic}\ \theta\in\mathrm{Irr}({T^{F^m}})\}}{\#\{\theta\in\mathrm{Irr}({T^{F^m}})\}}=1.
\end{equation*}
\end{prop}

\begin{proof}
First note that any element in $(G^l)^F$ can be written as $1+\pi^{l}x$ for a unique $x\in M_n(\mathcal{O}_{l'})$. Fix a non-trivial homomorphism of abelian groups $\psi\colon\mathcal{O}_l\rightarrow\overline{\mathbb{Q}}_{\ell}^{\times}$; then for all $x\in M_n(\mathcal{O}_{l'})$ we have
$$\widetilde{\theta}\mid_{(G^l)^F}(1+\pi^lx)=\psi(\mathrm{Tr}(\beta\cdot x))$$
for a unique $\beta\in\mathrm{Lie}(\mathbf{T})(\mathcal{O}_{l'})=\mathrm{Lie}(\mathbf{T})(\mathcal{O}^{\mathrm{ur}}_{l'})\cap M_n(\mathcal{O}_{l'})$. Now since $\psi$ is non-trivial, the regularity condition on $\theta$ can be reformulated as: There is a positive integer $a$ such that $\mathbf{T}^{\alpha}$ is defined over $\mathcal{O}^a_{l'}$ (here $\mathcal{O}^a$ denotes the unramified extension of $\mathcal{O}$ of degree $a$) for every $\alpha\in\Phi$, with the property
\begin{equation*}
\forall\alpha\in\Phi,\ \exists t\in\pi^{l'-1}\mathrm{Lie}(\mathbf{T}^{\alpha})(\mathcal{O}^{a}_{l'})\ \textrm{s.t.}\ \mathrm{Tr}(\beta\cdot N^{F^a}_F(t))\neq0,
\end{equation*}
where $N^{F^a}_F(t)$ is the norm of $t$ viewed as an element of $(T_{l'}^{l'-1})^{F^a}$. We claim that the regularity condition implies that
\begin{equation}\label{temp formula:regularity condition}
\forall\alpha\in\Phi,\ \exists t\in\pi^{l'-1}\mathrm{Lie}(\mathbf{T}^{\alpha})(\mathcal{O}^{\mathrm{ur}}_{l'})\ \textrm{s.t.}\ \mathrm{Tr}(\beta\cdot t)\neq0.
\end{equation}
Indeed, otherwise there is a root $\alpha$ such that $ \mathrm{Tr}(\beta\cdot t)=0$ for any $ t\in\pi^{l'-1}\mathrm{Lie}(\mathbf{T}^{\alpha})(\mathcal{O}^{\mathrm{ur}}_{l'})$; then, since $\beta$ is $F$-stable, we have 
$$\forall d\in\mathbb{Z}_{\geq0},\ \mathrm{Tr}(\beta\cdot F^d(t))=\mathrm{Tr}(F^d(\beta\cdot t))=F^d(\mathrm{Tr}(\beta\cdot t))=0,\ $$
which gives that $\mathrm{Tr}(\beta\cdot N^{F^a}_F(t))=0$ by the bilinearity of $\mathrm{Tr}(-,-)$ (note that, when viewing the $F^d(t)$'s as elements in $\pi^{l'-1}M_n(\mathcal{O}^{\mathrm{ur}}_{l'})$, the multiplication used in the definition of $N_F^{F^a}(t)$ becomes the matrix addition). So the claim holds. Now by conjugating $\mathbf{T}$ to be the diagonal torus, we can assume that $\beta$ is conjugated to  $\beta':=\mathrm{diag}(\beta_1,\cdots \beta_n)$ where $\beta_i\in\mathcal{O}^{\mathrm{ur}}_{l'}$, then \eqref{temp formula:regularity condition} can be reformulated as
\begin{equation}\label{temp formula: diagonal beta}
\pi\nmid\beta_i-\beta_j\ \forall i\neq j.
\end{equation}
As we are in $\mathrm{GL}_n$, by direct computations we see that \eqref{temp formula: diagonal beta} is equivalent to
\begin{equation}\label{temp formula: stabiliser condition 3}
C_{G_{l'}}(\beta)=T_{l'}.
\end{equation}
Meanwhile, taking quotients modulo $G^{l'}$, the stabiliser condition in Definition~\ref{defi: generic conditions}~(4) can be restated as
\begin{equation}\label{temp formula: stabiliser condition 1}
C_{G_{l'}^F}(\beta)=T_{l'}^F,
\end{equation}
which clearly follows from  \eqref{temp formula: stabiliser condition 3}. So the regularity condition implies the stabiliser condition. 

\vspace{2mm} To prove the equivalence between the regularity condition and the strongly generic condition, it now remains to show that the stabiliser condition implies the general position condition. Suppose that $w\in N_{G}(T)^F$ stabilises $\theta$, then it also stabilises $\widetilde{\theta}|_{(G^l)^F}$, so in terms of $\beta$ we have $\rho_{r,l'}(w)\in C_{G_{l'}^F}(\beta)$. By \eqref{temp formula: stabiliser condition 1} this implies that $\rho_{r,l'}(w)\in T_{l'}^F$. As the reduction map induces, for each $r'\in\mathbb{Z}_{[1,r]}$, an isomorphism between $N_{G_{r'}}(T_{r'})^F/T_{r'}^F$ and $N_{G_{1}}(T_{1})^F/T_{1}^F$ (see e.g.\ \cite[XXII~3.4]{SGA3}), we conclude that $w\in T^F$, i.e.\  $\theta$ is in general position.  So the regularity condition implies (and hence is equivalent to) the strongly generic condition. This proves the first assertion.

\vspace{2mm} For the second assertion, note that \eqref{temp formula: diagonal beta} can be written as 
\begin{equation}\label{temp formula: stabiliser condition 2}
\mathrm{Lie}(\alpha)(\beta)\in{(\mathcal{O}^{\mathrm{ur}}_{l'})}^{\times},\ \forall\alpha\in\Phi,
\end{equation}
where $\mathrm{Lie}(\alpha)\colon \mathrm{Lie}(\mathbf{T})(\mathcal{O}^{\mathrm{ur}}_{l'})\rightarrow \mathcal{O}^{\mathrm{ur}}_{l'}$ denotes the Lie algebra version of the root $\alpha$.  We claim that \eqref{temp formula: stabiliser condition 2} implies (and hence is equivalent to) the regularity condition. To see this, let $a$  be a positive integer such that all $\alpha,\check{\alpha}$ (and hence all $\mathbf{T}^{\alpha}$) are defined over $\mathbb{F}_{q^a}$. Then $\mathrm{Lie}(\alpha)(\beta)\in{(\mathcal{O}^{a}_{l'})}^{\times}$ for all $\alpha$. Note that $\pi^{l'-1}\mathrm{Lie}(\mathbf{T}^{\alpha})(\mathcal{O}^{\mathrm{a}}_{l'})$ can be viewed as the $a$-dimensional space $\mathbb{F}_{q^a}$ over $\mathbb{F}_{q}$, and every $t\in\pi^{l'-1}\mathrm{Lie}(\mathbf{T}^{\alpha})(\mathcal{O}^{\mathrm{a}}_{l'})$ corresponds to a unique $\tau\in\mathbb{F}_{q^a}$ (by conjugating $\mathbf{T}$ to be the diagonal torus) with the property that  $\mathrm{Tr}(\beta\cdot F^d(t))$, viewed as  an element in $\mathbb{F}_{q^a}\cong  \pi^{l'-1}\mathcal{O}^{a}_{l'}$, is equal to $\mathrm{Lie}(F^d(\alpha))(\beta)\cdot F^d(\tau)\pmod\pi$, for $d=0,1,...,a-1$. As $\beta\in\mathrm{Lie}(\mathbf{T})(\mathcal{O}_{l'})$, we have
$\mathrm{Lie}(F^d(\alpha))(\beta)\cdot F^d(\tau)=F^d(\mathrm{Lie}(\alpha)(\beta)\cdot \tau)$, so
$$\sum_{d=0}^{a-1}\mathrm{Lie}(F^d(\alpha))(\beta)\cdot F^d(\tau)
=\sum_{d=0}^{a-1}F^d(\mathrm{Lie}(\alpha)(\beta)\cdot \tau).$$
In the right hand side, note that, after modulo $\pi$, the function $\sum_{d}F^d(-)$ is the trace of the field extension $\mathbb{F}_{q^a}/\mathbb{F}_q$; it is well-known that this function is not identically zero since $\mathbb{F}_{q^a}/\mathbb{F}_q$ is separable. Thus, by \eqref{temp formula: stabiliser condition 2}, there is a $\tau\in\mathbb{F}_{q^a}$  such that 
$$\sum_{d=0}^{a-1}F^d(\mathrm{Lie}(\alpha)(\beta)\cdot \tau) \not\equiv 0 \pmod \pi,$$
that is, $\mathrm{Tr}(\beta\cdot N^{F^a}_F(t))\neq0$ for the $t$ corresponding to $\tau$. So the claim holds. Meanwhile, note that \eqref{temp formula: stabiliser condition 2} is equivalent to
$$\mathrm{Lie}(\alpha)(\beta)\not\equiv 0 \pmod\pi,\ \forall\alpha\in\Phi,$$
that is, $\beta\pmod\pi\in\mathrm{Lie}(\mathbf{T})(\overline{\mathbb{F}}_q)\subseteq M_n(\overline{\mathbb{F}}_q)$ is regular. Thus the regularity of $\theta$ is equivalent to the regularity of $\beta\pmod\pi$. In the below we give a lower bound of the number of the regular elements in $\mathrm{Lie}(\mathbf{T})(\overline{\mathbb{F}}_q)$. 

\vspace{2mm} As $\mathbb{G}=\mathrm{GL}_n$, $\beta\pmod\pi$ is regular if and only if $\beta\pmod\pi$ has mutually different eigenvalues.  Considering eigenvalues of elements in $M_n(\overline{\mathbb{F}}_q)$, we see that the irregular semisimple elements of the Lie algebra $\mathrm{Lie}(\mathbf{T})_{\overline{\mathbb{F}}_q}$  form a hyperplane arrangement $Y$ (in the affine space $\mathrm{Lie}(\mathbf{T})_{\overline{\mathbb{F}}_q}$) with $\binom{n}{2}$ irreducible components (each component determined by a root equation $\mathrm{Lie}(\alpha)(-)\equiv0\pmod\pi$). Since $\mathrm{Lie}(\mathbf{T})_{\overline{\mathbb{F}}_q}$ can be conjugated to the diagonal Cartan subalgebra by a single element in $\mathrm{GL}_n(\overline{\mathbb{F}}_q)$, $\mathrm{Lie}(\mathbf{T})_{\overline{\mathbb{F}}_q}$ is cut out by $n(n-1)$ independent linear homogeneous equations from the affine space $M_n\cong\mathbb{A}^{n^2}$. So each irreducible component $Y_i$ of $Y$ is cut out by $n(n-1)+1$ independent linear homogeneous equations from $M_n$; in particular, basic linear algebra tells that, for each $i$, the number of the matrices in $M_n(\mathbb{F}_{q^m})\cap Y_i$  is $\leq q^{m\cdot (n^2-n(n-1)-1)}=q^{m\cdot(n-1)}$, so $\#Y(\mathbb{F}_{q^m})\leq\binom{n}{2}\cdot q^{m\cdot(n-1)}$.
Meanwhile, since $\mathrm{Lie}(\mathbf{T})_{{\mathbb{F}}_q}$ is an affine space, we have $\#\mathrm{Lie}(\mathbf{T})(\mathbb{F}_{q^m})=q^{m\cdot n}$ (see e.g.\ \cite[Proposition~10.11]{DM1991}).

\vspace{2mm} Now note that, for any $\beta\in\mathrm{Lie}(\mathbf{T})(\mathcal{O}_{l'})$, the representation $\mathrm{Ind}_{(T^{l})^F}^{T^F}\psi(\mathrm{Tr}(\beta\cdot(-))$ is multiplicity-free of dimension $\#T_l^F$, a number independent of $\beta$, with each irreducible constituent being of dimension $1$; it follows that the number of the $\theta$'s producing the same $\beta$ is independent of $\beta$, which implies that\begin{equation*}
\begin{split}
\lim_{m\rightarrow+\infty}\frac{\#\{\textrm{strongly generic}\ \theta\in\mathrm{Irr}({T^{F^m}})\}}{\#\{\theta\in\mathrm{Irr}({T^{F^m}})\}}
=&
\lim_{m\rightarrow+\infty}\frac{\#\{\textrm{regular elements in}\ \mathrm{Lie}(\mathbf{T})(\mathbb{F}_{q^m})\}}{\#\mathrm{Lie}(\mathbf{T})(\mathbb{F}_{q^m})}\\
\geq &
\lim_{m\rightarrow+\infty}\frac{q^{m\cdot n}-\binom{n}{2}\cdot q^{m\cdot(n-1)}}{q^{m\cdot n}}= 1,
\end{split}
\end{equation*}
as desired.
\end{proof}

From \cite{Lusztig2004RepsFinRings}, \cite{Sta2009Unramified} we know that, if $\theta$ is generic, then $R_{T,U}^{\theta}$ is (up to a sign) irreducible. 

\vspace{2mm} In the remaining part of this section let $r>1$. We shall always take the viewpoint that $G^{r-1}$ is the additive group of the Lie algebra $\mathfrak{g}$ of $G_1$ (see \cite[II~4.3]{Demazure_Gabriel_1980intro_AG_AG}); note that $G_1=G/G^{1}$  acts on $G^{r-1}$ through 
 the conjugation action of $G$. For any irreducible $\overline{\mathbb{Q}}_{\ell}$-representation $\sigma$ of $G^F$, by Clifford's theorem  we have  (see e.g.\ \cite[6.2]{Isaacs_CharThy_Book})
$$\sigma|_{(G^{r-1})^F}\cong e\cdot\left(\sum_{\chi\in G_1^F\cdot \sigma} \chi \right),$$
where $e\in \mathbb{Z}_{>0}$ and $G_1^F\cdot \sigma$ denotes the $G_1^F$-orbit of irreducible representations of $\mathfrak{g}^F$ (viewed as an additive group), which gives a map
\begin{equation*}
\mathrm{Irr}(G^F)\longrightarrow G_1^F\backslash\mathrm{Irr}(\mathfrak{g}^F).
\end{equation*}

\vspace{2mm} Now consider the assumption:
\begin{itemize}\hypertarget{condition (A)}{}
\item[{\bf (A)}] There exists a non-degenerate symmetric bilinear form
$$\mu(-,-)\colon G^{r-1}\times G^{r-1}\longrightarrow \overline{\mathbb{F}}_q$$ 
defined over $\mathbb{F}_q$, such that the following property \hyperlink{condition (A1)}{{(A1)}} holds:

\item[]\hypertarget{condition (A1)}{}
\subitem (A1)  $\mu$ is invariant under the conjugation of $G_1$ and the adjoint action of $\mathfrak{g}$ (the latter means that $\mu([x,y],z)=\mu(x,[y,z])$ where $[,]$ is the Lie bracket).
\end{itemize}

\vspace{2mm} Note that \hyperlink{condition (A1)}{{(A1)}} implies that: (See \cite[2.5.1]{Let2005book}) 
\begin{itemize}
\item[]\hypertarget{condition (A2)}{} 
\subitem (A2) $\mu(U_{\alpha}^{r-1},T^{r-1}\prod_{\beta\in\Phi\backslash\{-\alpha\}}U_{\beta}^{r-1})=0$ for any root $\alpha$. 
\end{itemize}

\vspace{2mm} In the remaining part of this paper, whenever we talk about the notion of orbits of irreducible representations, we shall work under the assumption \hyperlink{condition (A)}{{(A)}} and fix a choice of such a form. When does such a form exist? If $\mathbb{G}=\mathrm{GL}_n$, then one can just take the trace form; for a general $\mathbb{G}$, one can again take  trace form if $p$ is a very good prime (see \cite[I.5.3]{Springer--Steinberg_1970_Conj}). For further details on such bilinear forms we refer to \cite[2.5]{Let2005book}. 

\vspace{2mm} We shall also fix a non-trivial character $\psi\colon \mathbb{F}_q\rightarrow \overline{\mathbb{Q}}_{\ell}^{\times}$: So for any $y\in \mathfrak{g}^F\cong (G^{r-1})^F$ we have a character 
$$\psi_y\colon x\longmapsto \psi(\mu(x,y))$$ 
of $\mathfrak{g}^F$. Note that the non-degeneracy of $\mu$ implies that $\psi_{(-)}$ is an isomorphism from $\mathfrak{g}^F$ to $\mathrm{Irr}(\mathfrak{g}^F)$, thus combining $\psi$ with the above map $\mathrm{Irr}(G^F)\rightarrow G_1^F\backslash\mathrm{Irr}(\mathfrak{g}^F)$ we get another map
$$\Omega\colon \mathrm{Irr}(G^F)\longrightarrow G_1^F\backslash \mathfrak{g}^F.$$

\begin{defi}\label{defi: orbit rep}
An irreducible representation $\sigma$ of $G^F$ is called \emph{regular} (resp.\ \emph{nilpotent}, \emph{semisimple}), if $\Omega(\sigma)$ is regular (resp.\ nilpotent, semisimple).
\end{defi}

In Theorem~\ref{thm:orbit type} we will prove that the orbit of any generic higher Deligne--Lusztig representation is regular and semisimple.

\vspace{2mm} Some conventions: 
\begin{itemize}
\item In the remaining part of this paper, unless otherwise specified, $r>1$.
\item For $a,b$ in a  group, we often use the notation $a^b:=b^{-1}ab=:{^{b^{-1}}a}$.
\item We use $H_c^*(-)$ short for $\sum_i(-1)^iH_c^i(-,\overline{\mathbb{Q}}_{\ell})$.
\item If $\sigma$ is a virtual representation, then we denote by $|\sigma|$ the corresponding true representation (by taking absolute values of the coefficients of all irreducible constituents).
\end{itemize}

\section{Induction, orbits, and dimensions}\label{sec:main strategy}

When $r$ is even, we obtained the following algebraisation theorem in \cite{ChenStasinski_2016_algebraisation}:

\begin{thm}\label{thm:main_even}
Suppose that $r$ is even. Then 
$$R_{T,U}^{\theta}\cong\mathrm{Ind}_{(TG^{l'})^F}^{G^F}\widetilde{\theta}$$
if $\theta\in \mathrm{Irr}({T^F})$ is strongly generic.
\end{thm}

The fundamental idea in approaching Theorem~\ref{thm:main_even} is to realise the algebraically constructed representations via cohomology, and then compare them with $R_{T,U}^{\theta}$, as achieved in \cite{ChenStasinski_2016_algebraisation}. The main task of this paper is to generalise this theorem for an arbitrary $r>1$ (Theorem~\ref{thm:full alg}). 

\vspace{2mm} However, when $r$ is odd there are two additional non-trivial extensions of representations involved (see Section~\ref{sec:alg construction}), which make the situation considerably more difficult. To overcome the problem, our first step is to  pass from $G^F$ to $(TG^{l'})^F$ via an induction theorem; for this we need some extra inputs:

\vspace{2mm} Let $U^{l',l}:=U^{l'}(U^-)^{l}$.

\begin{prop}\label{prop:middle level step}
If $\theta$ is regular, then  
$$\langle R_{T,U}^{\theta}, H_c^*(L^{-1}(FU^{l',l}))_{\theta}\rangle_{G^F}=\# \mathrm{Stab}_{W(T)^F}(\theta).$$
\end{prop}
\begin{proof}
This is a special case of \cite[2.4]{Chen_2017_InnerProduct}.
\end{proof}

\begin{lemm}\label{lemm:permut induced rep}
Let $X_0:=L^{-1}(FU^{l',l})\cap TG^{l'}$. Then 
$$L^{-1}(FU^{l',l})=\coprod_{g\in G^F/(TG^{l'})^F}gX_0.$$
Moreover, as representations of $G^F$, we have
$$H^*_c(L^{-1}(FU^{l',l}))_{\theta}=\mathrm{Ind}_{(TG^{l'})^F}^{G^F} H_c^*(X_0)_{\theta},$$
for any $\theta\in\mathrm{Irr}(T^F)$.
\end{lemm}
\begin{proof}
The second assertion follows from the first assertion and \cite[1.7]{Lusztig_whiteBk}. For the first assertion, first note that the $\coprod_{g\in G^F/(TG^{l'})^F}gX_0$ is by construction a subset of the $L^{-1}(FU^{l',l})$. On the other hand, if $x\in L^{-1}(FU^{l',l})$, then $\rho_{r,l'}(x)^{-1}F(\rho_{r,l'}(x))=1$ which implies that $\rho_{r,l'}(x)\in G_{l'}^F$. So for a representative $g\in G^F$ (of $\rho_{r,l'}(x)$) one has $g^{-1}x\in L^{-1}(FU^{l',l})\cap G^{l'}\subseteq X_0$; this shows that the $L^{-1}(FU^{l',l})$ is also a subset of the $\coprod_{g\in G^F/(TG^{l'})^F}gX_0$.
\end{proof}

We also need to know some properties of $H_c ^*(X_0)_{\theta}$:

\begin{lemm}\label{lemm: statement (1)}
If $\theta$ is regular, then
$$\langle H_c ^*(X_0)_{\theta},H_c ^*(X_0)_{\theta}\rangle_{(TG^{l'})^F}=1.$$
\end{lemm}

The argument of this lemma is a bit technical, and uses a modified version of a theme initiated by Lusztig \cite{Lusztig2004RepsFinRings} and continued in \cite{ChenStasinski_2016_algebraisation}; we will treat it in Section~\ref{sec: tech1}.

\begin{lemm}\label{lemm: statement (2)}
For any $\theta\in \mathrm{Irr}({T^F})$, $H_c ^*(X_0)_{\theta}|_{(G^l)^F}$ is a multiple of $\widetilde{\theta}|_{(G^l)^F}$, where $\widetilde{\theta}$ is as in Definition~\ref{defi: generic conditions}~(4). Moreover, if $\theta$ is regular, then this multiple is non-zero.
\end{lemm}

\begin{proof}
Since $[G^i,G^j]\subseteq G^{i+j}$, $U^{\pm}$ commutes with every element in $G^{l'}$. Thus by $L(U^{\pm})\subseteq U^{\pm}=FU^{\pm}\subseteq FU^{l',l}$,  the left action of $(U^{\pm})^F\subseteq (TG^l)^F$ on $X_0$ naturally extends to that of the connected group $U^{\pm}$, so the left action of $(U^{\pm})^F$ on each cohomology group $H_c^i(X_0)$ is actually trivial (see e.g.\ \cite[10.13]{DM1991}). Meanwhile, if $s\in (T^l)^F$, then $s$ commutes with every element in $X_0$, so the left $(T^l)^F$-module structure of $H^*_c(X_0)_{\theta}|_{(T^l)^F}$ is the same as its right $(T^l)^F$-module structure, which is by definition a multiple of $\theta|_{(T^l)^F}$. The ``non-zero'' assertion then follows from Lemma~\ref{lemm: statement (1)}. This completes the proof.
\end{proof}

We can now realise $R_{T,U}^{\theta}$, for $\theta$ strongly generic, as induced representation:

\begin{thm}\label{thm:induction theorem}
Suppose that $\theta$ is strongly generic. Then
$$R_{T,U}^{\theta}\cong\mathrm{Ind}_{(TG^{l'})^F}^{G^F} H_c^*(X_0)_{\theta}.$$
\end{thm}
\begin{proof}
By the assumption on $\theta$, $\mathrm{Stab}_{G^F}(\widetilde{\theta}|_{(G^l)^F})=(TG^{l'})^F$. So by Clifford theory, Lemma~\ref{lemm: statement (1)} and Lemma~\ref{lemm: statement (2)} imply that, up to a sign, the representation $\mathrm{Ind}_{(TG^{l'})^F}^{G^F} H_c^*(X_0)_{\theta}$  is irreducible. Meanwhile, $\mathrm{Ind}_{(TG^{l'})^F}^{G^F} H_c^*(X_0)_{\theta}$ is isomorphic to $H_c^*(L^{-1}(FU^{l',l}))_{\theta}$ by Lemma~\ref{lemm:permut induced rep}. So the assertion follows from  Proposition~\ref{prop:middle level step} and the irreducibility of $R_{T,U}^{\theta}$.
\end{proof}

Note that, for $r$ even, the orbit $\Omega(R_{T,U}^{\theta})$ for a strongly generic $\theta$ is regular semisimple; this is a consequence of the even level algebraisation Theorem~\ref{thm:main_even}. Now, without a priori having an algebraisation theorem (and without the strong genericity condition but only assuming $\theta$ is regular), we can deduce the result for all $r>1$:

\begin{thm}\label{thm:orbit type}
Suppose that $\theta$ is regular. Then the orbit of 
at least one irreducible constituent of $|R_{T,U}^{\theta}|$ is regular and semisimple. In particular, if $\theta$ is generic, then $\Omega(|R_{T,U}^{\theta}|)$ is regular and semisimple.
\end{thm}

\begin{proof}
By Proposition~\ref{prop:middle level step} and Frobenius reciprocity, we know that, up to a sign $R_{T,U}^{\theta}$ contains $H_c^*(X_0)_{\theta}$ as an irreducible constituent (Lemma~\ref{lemm: statement (1)}), after restricting to $(TG^{l'})^F$. Let $s\in (G^{r-1})^F$ be the element dual to $\widetilde{\theta}|_{(G^{r-1})^F}$ via the bilinear form $\mu(-,-)$ on $G^{r-1}$. Then by Lemma~\ref{lemm: statement (2)} it is enough to show that $s$ is regular and semisimple (as an element in the Lie algebra $G^{r-1}=\mathfrak{g}$).

\vspace{2mm} First, one easily checks that $s\in (T^{r-1})^F$ using 
\hyperlink{condition (A2)}{{(A2)}}, so $s$ is semisimple. We shall show that the regularity of $\theta$ implies the regularity of $s$.  Indeed, otherwise we would have ${\alpha}(s)=0$ for some root ${\alpha}$. For one such $\alpha$ choose $e_{\alpha}\in U_{\alpha}^{r-1}$ and $e_{-\alpha}\in U_{-\alpha}^{r-1}$ with $h_{\alpha}:=[e_{\alpha},e_{-\alpha}]$ generating $\mathcal{T}^{\alpha}$ as a vector space, where $[,]$ denotes the Lie bracket on $G^{r-1}=\mathfrak{g}$. Hence $$0=\mu(\alpha(s)e_{\alpha},e_{-\alpha})=\mu([s,e_{\alpha}],e_{-\alpha})=\mu(s,h_{\alpha})$$
(for the last equality see 
\hyperlink{condition (A1)}{{(A1)}}). This implies that $\mu(s,\mathcal{T}^{\alpha})=0$, hence the bilinearity gives that $\mu(s,N^{F^a}_F(\mathcal{T}^{\alpha}))=0$ (recall that $s$ is $F$-stable and $\mu$ is defined over $\mathbb{F}_q$), which contradicts the regularity of $\theta$, thus $s$ must be regular. The assertion follows.
\end{proof}

\begin{remark}\label{remark:strongly generic at r=2,3}
Suppose that the assumption \hyperlink{condition (A)}{{(A)}} holds and  the centralisers in $G_1$ of regular elements in $T^{r-1}$ are connected. Then in the above proof of Theorem~\ref{thm:orbit type}, since the centraliser of $s$ in $G^F$ is the same with the centraliser of $\widetilde{\theta}|_{(G^{r-1})^F}$ in $G^F$, we see that the regularity of $\theta$ (which implies the regularity of $s$) gives  $\mathrm{Stab}_{G^F}(\widetilde{\theta}|_{(G^{r-1})^F})=(TG^{1})^F$. In particular, if $r=2$ or $3$ (so we have $l'=1$), then being generic is equivalent to being strongly generic.
\end{remark}

Note that Theorem~\ref{thm:orbit type} gives an affirmative answer to the problem posed in \cite{Stasinski_Stevens_2016_regularRep}. Meanwhile, it can be viewed as an $r>1$ analogue of the following  property: For $r=1$, if $\theta$ is in general position, then $|R_{T,U}^{\theta}|$ is both regular and semisimple in the sense of representation theory of finite groups of Lie type (for the meaning of these two notions, see e.g.\ \cite[2.6.9 and 3.4.13]{Geck_Malle_2020book}; the asserted property follows from \cite[2.6.10 and 3.4.4]{Geck_Malle_2020book}). 

\vspace{2mm} Conversely, for $\mathrm{GL}_n$, Hill \cite[Section~2. Remarks]{Hill_1993_Jordan} claimed that his geometric conjugacy classes (which are unions of orbits in the sense of our paper) can be described by an analogue of Deligne--Lusztig characters; this is just like the $r=1$ case in the scope of Jordan decomposition. We solve the regular semisimple case of this claim:

\begin{prop}\label{coro:dim formula}
Let $\mathbb{G}=\mathrm{GL}_n$, $r>1$, and $\sigma\in\mathrm{Irr}(G^F)$. Then $\Omega(\sigma)$ is regular and semisimple if and only if $\sigma\cong|R_{T,U}^{\theta}|$ for some pair $(T,\theta)$ with $\theta$ being regular. In addition, for every regular $\theta$  we have the dimension formula
\begin{equation*}
\dim |R_{T,U}^{\theta}|=\frac{|G_1^F|}{|T_1^F|}\cdot q^{(r-2)\cdot\frac{n^2-n}{2}}= {|G_1^F/T_1^F|_{p'}}\cdot q^{(r-1)\cdot\frac{n^2-n}{2}}.
\end{equation*}
\end{prop}

The dimension formula part is actually a special case of Corollary~\ref{coro:dim formula for general case}, which works for an arbitrary reductive group, not just $\mathrm{GL}_n$. However, this latter general result requires Theorem~\ref{thm:full alg}, and here for $\mathrm{GL}_n$ we can give a simpler proof that does not need Theorem~\ref{thm:full alg}.
\begin{proof}
Let $s\in (T^{r-1})^F$ be regular. We view $s$ as an element in $M_n(\mathbb{F}_q)$, and take $\hat{s}$ to be a lift in $M_n(\mathcal{O}_r)$, following \cite[Page-3536]{Hill_1993_Jordan}. Since $s$ is regular, we have $C_G(\hat{s})=T$. So \cite[Theorem~2.13]{Hill_1993_Jordan} tells that there is a bijection between the sets
$$A:=\{\sigma'\in\mathrm{Irr}(G^F) \mid \Omega(\sigma')=(s)\}$$
and
$$B:=\{\theta'\in\mathrm{Irr}(T^F)(=\mathrm{Irr}(C_G(\hat{s})^F))\mid \Omega(\theta')=(0)\}.$$
Given $\theta\in\mathrm{Irr}(T^F)$, if $\theta$ corresponds to $s$ via the trace form (as in the proof of Proposition~\ref{prop: regular in GL_n}), then $\theta$ is regular. Hence by Proposition~\ref{prop: regular in GL_n} and the argument of Theorem~\ref{thm:orbit type}, we see that $|R_{T,U}^{\theta}|$ is irreducible and $\Omega(|R_{T,U}^{\theta}|)=(s)$. Now note that, since $T^F$ is abelian, the number of $\theta$ corresponding to $(s)$ and the cardinal of $B$ are both $|T_{r-1}^F|$. So any element of $A$ is of the form $|R_{T,U}^{\theta}|$ for a regular $\theta$. This proves the first assertion. As any element of $B$ is of $\dim=1$, the dimension formula now follows from the ratio formula in \cite[Theorem~2.13]{Hill_1993_Jordan}.
\end{proof}

For regular $s\in T_1^F$, in \cite[Theorem~4.1]{Chen_2018_GreenFunction} we obtained a character formula for $R_{T,U}^{\theta}$ at $s$, when $\mathrm{char}(\mathcal{O})>0$. Here, as a quick application of Theorem~\ref{thm:induction theorem}, we deduce a new proof of this formula for general $\mathcal{O}$, under the strongly generic condition:

\begin{coro}\label{coro:character formula at rss}
Suppose that $s\in T_1^F$ is a regular semisimple element. If $\theta$ is strongly generic, then
$$\mathrm{Tr}(s,R_{T,U}^{\theta})=\sum_{w\in W(T)^F}(^w\theta)(s).$$
\end{coro}
\begin{proof}
Given an element $t\in T^F$, denote its semisimple part by $t'\in T_1$ and its unipotent part by $t''\in T^1$. Then by Brou\'e's formula for generalised bi-module induction (see \cite[Proposition~4.5]{DM1991}) we have
\begin{equation}\label{formula: basic formula for T_1^F}
\begin{split}
\mathrm{Tr}(s,H^*_c(X_0)_{\theta})
&=\frac{1}{|T^F|}\sum_{t\in T^F}\theta(t^{-1})\cdot\mathrm{Tr}((s,t)\mid H_c^*(X_0))\\
&=\frac{1}{|T^F|}\sum_{t\in T^F}\theta(t^{-1})\cdot\mathrm{Tr}((1,t'')\mid H_c^*(X_0^{(s,t')})),
\end{split}
\end{equation}
where $X_0^{(s,t')})$ denotes the subvariety consisting of the elements fixed by $(s,t')$, and where the second equality follows from Deligne--Lusztig's fixed point formula (\cite[Theorem~3.2]{DL1976}). Note that, by construction $X_0^{(s,t')}$ is non-empty only if $t'=s^{-1}$, so
\begin{equation}\label{formula: basic formula for T_1^F 2}
\begin{split}
\eqref{formula: basic formula for T_1^F}
&=\frac{1}{|T^F|}\sum_{t''\in (T^1)^F}\theta(t^{-1})\cdot\mathrm{Tr}((1,t'')\mid H_c^*(X_0^{(s,s^{-1})}))\\
&=\theta(s)\cdot\frac{1}{|T^F|}\sum_{t''\in (T^1)^F}\theta({t''}^{-1})\cdot\mathrm{Tr}((1,t'')\mid H_c^*(X_0\cap C_{TG^{l'}}(s)))\\
&=\theta(s)\cdot\frac{1}{|T_1^F|} \dim H_c^*(X_0\cap C_{TG^{l'}}(s))_{\theta|_{(T^1)^F}},
\end{split}
\end{equation}
where the last equality follows again from Brou\'e's formula.

\vspace{2mm} As our $s$ is regular, $C_{TG^{l'}}(s)^{\circ}=T$. We claim that $C_{TG^{l'}}(s)=C_{TG^{l'}}(s)^{\circ}$ (this holds for any $s\in T_1$, not necessarily regular). To see this, first note that (via the Iwahori decomposition  $TG^{l'}=T\times\prod_{\alpha\in\Phi}U_{\alpha}^{l'}$) 
$$C_{TG^{l'}}(s)=\langle T,U_{\alpha}^{l'}\mid \alpha(s)=1  \rangle.$$
However, $T$ and $U_{\alpha}$ (with $\alpha(s)=1$) are subgroups of $C_{TG^{l'}}(s)^{\circ}$ (since they are connected groups), so the claim holds. Now we see that $X_0\cap C_{TG^{l'}}(s)=L^{-1}(FU^{l',l})\cap T=T^F$, so $\eqref{formula: basic formula for T_1^F 2}=\theta(s)$. Applying Theorem~\ref{thm:induction theorem} to this equality we get the assertion by the formula for induced characters.
\end{proof}

\begin{remark}\label{remark: C_G(s)}
Let $s\in T_1$ be a general element (not necessarily regular). In the proof of Corollary~\ref{coro:character formula at rss} we see that  $C_{TG^{l'}}(s)=C_{TG^{l'}}(s)^{\circ}=\langle T,U_{\alpha}^{l'}\mid \alpha(s)=1  \rangle$.  In particular, we have $C_{TG^{l'}}(s)\subseteq C_{G}(s)^{\circ}$. It shall also be useful to note that 
$$C_{G}(s)^{\circ}=\langle T,U_{\alpha}\mid \alpha(s)=1  \rangle,$$
which is well-known in the classical $r=1$ case. For $r>1$ this property is proved for $\mathrm{char}(\mathcal{O})>0$, implicitly, in \cite{Chen_2018_GreenFunction}. For general $\mathcal{O}$, first note that $\langle T,U_{\alpha}\mid \alpha(s)=1  \rangle\subseteq C_{G}(s)^{\circ}$ by the same argument in the proof of Corollary~\ref{coro:character formula at rss}. Meanwhile, given $g\in C_{G}(s)^{\circ}$, since the reduction map $\rho_{r,1}$ takes $\langle T,U_{\alpha}\mid \alpha(s)=1  \rangle$ surjectively onto $C_{G_1}(s)^{\circ}=\langle T_1,(U_{\alpha})_1\mid \alpha(s)=1  \rangle$, there is a $g'\in \langle T,U_{\alpha}\mid \alpha(s)=1  \rangle$ satisfying that $\rho_{r,1}(g)=\rho_{r,1}(g')$, which implies that 
$$g'g^{-1}\in G^1\cap C_{G}(s)=C_{G^1}(s)=\langle T^1,U_{\alpha}^1\mid \alpha(s)=1  \rangle,$$
where the last equality follows from the Iwahori decomposition $G^1=T^1\times\prod_{\alpha}U_{\alpha}^1$. So $g'g^{-1}\in \langle T,U_{\alpha}\mid \alpha(s)=1  \rangle$, and hence $g\in\langle T,U_{\alpha}\mid \alpha(s)=1  \rangle$.
\end{remark}

\section{On comparing the algebraic constructions}\label{sec:alg construction}

In this section we explain how the induction formula in Theorem~\ref{thm:induction theorem} leads to an algebraisation theorem for $R_{T,U}^{\theta}$ for all $r>1$.

\vspace{2mm} In this section let $\widetilde{\theta}$ be the lift of $\theta$ to $(TG^l)^F=(TU^{\pm})^F$ along $TU^{\pm}\rightarrow T^F$. Note that $(T^1G^l)^F$ is a normal subgroup of $(T^1G^{l'})^F$, and $\widetilde{\theta}|_{(T^1G^l)^F}$ is of dimension one; in such a situation the theory of Heisenberg lifts is available (see \cite[Section~3]{Stasinski_Stevens_2016_regularRep} and \cite[Section~8.3]{Bushnell_Froelich_book_GaussSum_1983}), which gives that:

\begin{lemm}\label{lemm:existence of Heisenberg lift}
Let $r>1$ be odd, and suppose that $\theta$ satisfies the stabiliser condition in Definition~\ref{defi: generic conditions}, namely $\mathrm{Stab}_{G^F}(\widetilde{\theta}|_{(G^l)^F})=(TG^{l'})^F$. Then there is a unique irreducible representation $\rho_{\theta}$ of $(T^1G^{l'})^F$ containing $\widetilde{\theta}|_{(T^1G^l)^F}$. Moreover, 
$$\mathrm{Ind}_{(T^1G^{l})^F}^{(T^1G^{l'})^F}\widetilde{\theta}|_{(T^1G^l)^F}=(\dim\rho_{\theta})\rho_{\theta};$$
where $\dim \rho_{\theta}=q^{\#\Phi^+}$. This $\rho_{\theta}$ is called the Heisenberg lift of $\widetilde{\theta}|_{(T^1G^l)^F}$.
\end{lemm}
(There is a conceptual explanation of the dimension value $q^{\#\Phi^+}$: The Steinberg representation appears here; see Lemma~\ref{lemm:full alg of H(X_0)}.)
\begin{proof}
This follows from \cite[Lemma~3.3]{Stasinski_Stevens_2016_regularRep}.
\end{proof}

\begin{prop}\label{prop:existence of ext of rho'}
The above $\rho_{\theta}$ admits extensions to $(TG^{l'})^F$. 
\end{prop}

\begin{proof}
Since  the cardinals of $(T^1G^{l'})^F$ and of $T_1^F=(TG^{l'})^F/(T^1G^{l'})^F$ are coprime, it suffices to show that $T_1^F$ stabilises $\rho_{\theta}$ (see e.g.\ \cite[Corollary~6.2]{Navarro_2018_book_McKayConj}). Indeed, because $T_1^F$ stabilises $\widetilde{\theta}|_{(T^1G^l)^F}$, the uniqueness of $\rho_{\theta}$ in Lemma~\ref{lemm:existence of Heisenberg lift} implies that $T_1^F$ stabilises $\rho_{\theta}$. 
\end{proof}

When $\theta$ is strongly generic, as will be shown in the below, $H_c^*(X_0)_{\theta}$ is actually an extension of $\rho_{\theta}$, We shall determine its character explicitly, except for $q<7$, in which case we determine it up to sign.

\begin{lemm}\label{lemm:full alg of H(X_0)}
Let $r>1$ be odd. We have:
\begin{itemize}
\item[(i)] Suppose that $\theta$ is regular. The restriction of  $H_c ^*(X_0)_{\theta}$ to $(T^1)^F$ is a non-zero multiple of $\theta|_{(T^1)^F}$. Furthermore, if  $\theta$  satisfies the stabiliser condition $\mathrm{Stab}_{G^F}(\widetilde{\theta}|_{(G^l)^F})=(TG^{l'})^F$, then $|H_c ^*(X_0)_{\theta}|$ is isomorphic to an extension  of $\rho_{\theta}$, from $(T^1G^{l'})^F$ to $(TG^{l'})^F$.

\item[(ii)] Suppose that $\theta$ is strongly generic.  Then there is a unique sign function $e_{\theta}\colon T_1^F\rightarrow \{1,-1\}$, such that for every $s\in T_1^F$ one has 
\begin{equation*}
\begin{split}
\mathrm{Tr}(s, H_c ^*(X_0)_{\theta})
&=e_{\theta}(s)\cdot q^{\#\Phi_s^{+}}\theta(s)\\
&=(-1)^{\mathrm{rk}_q{(G_1)}+\mathrm{rk}_q{(C_{G_1}(s)^{\circ}})}e_{\theta}(s)\cdot\mathrm{St}_{G_1}(s)\theta(s),
\end{split}
\end{equation*}
where $\mathrm{St}_{G_1}$ denotes the Steinberg character, $\mathrm{rk}_q(-)$ denotes the $\mathbb{F}_q$-rank (relative $F$-rank) of algebraic group, and $\Phi_s^{+}$ denotes a set of positive roots (with respect to $T$) of $C_{G}(s)^{\circ}$. 

\item[(iii)] Suppose that $\theta$ is strongly generic. Furthermore, suppose that $q\geq7$. Then
$$e_{\theta}(s)=(-1)^{\mathrm{rk}_q{(T_1)}+\mathrm{rk}_q({C_{G_1}(s)^{\circ}})};$$ 
in particular we have $|H_c^*(X_0)_{\theta}|=(-1)^{\mathrm{rk}_q({G_1})+\mathrm{rk}_q({T_1})}H_c^*(X_0)_{\theta}$.
\end{itemize}
\end{lemm}
\begin{proof}
(i) First, by a general property of $\ell$-adic cohomology (see e.g.\ \cite[Proposition~10.12]{DM1991}) we know that 
$$H_c ^*(X_0)_{\theta}\cong H_c ^*(X_0/U^{\pm})_{\theta},$$
which allows us to focus on the variety $X_0/U^{\pm}$. 

\vspace{2mm} We shall show that, if $t\in T^1$, then $t$ commutes with every element in $X_0/U^{\pm}$; it suffices to show that, for every $u\in U^{l'}$ and $v\in (U^-)^{l'}$ one has $tuvU^{\pm}=uvU^{\pm}t$, namely, $(uv)^{-1}t^{-1}(uv)t\in U^{\pm}$. Indeed, for every $w\in G^{l'}$ one has that (since $[G^i,G^j]\subseteq G^{i+j}$)
$$[w,t]:=wtw^{-1}t^{-1}\in G^l,$$
so
$$(uv)^{-1}t^{-1}(uv)t  =   v^{-1}   [u^{-1},t^{-1}]  v^t =  [u^{-1},t^{-1}] [v^{-1},t^{-1}]  ,     $$
where the second equality follows from the commutativity between $[u^{-1},t^{-1}]\in G^l$ and $v^{-1}\in G^{l'}$. Since $t$ stabilises $U$, we get that $[u^{-1},t^{-1}]\in G^{l}\cap U$, and similarly $[v^{-1},t^{-1}] \in G^{l}\cap U^-$, thus $(uv)^{-1}t^{-1}(uv)t\in U^{\pm}$ as desired.

\vspace{2mm} The above commutativity implies that, for every $t\in (T^1)^F$ and every $i\in\mathbb{Z}$, the two automorphisms of $H_c^i(X_0/U^{\pm})$ induced by the left translation of $t$ and the right translation of $t$, respectively, actually  coincide. Therefore
$$H_c ^*(X_0)_{\theta}|_{(T^1)^F}\cong H_c ^*(X_0/U^{\pm})_{\theta}|_{(T^1)^F}=m{\theta}|_{(T^1)^F}$$
as left $(T^1)^F$-modules, for some $m\in\mathbb{Z}$; this $m$ is non-zero by Lemma~\ref{lemm: statement (1)}. 

\vspace{2mm} Now, by Lemma~\ref{lemm: statement (2)}, we see that $H_c ^*(X_0)_{\theta}|_{(T^1G^l)^F}$ is a non-zero multiple of $\widetilde{\theta}|_{(T^1G^l)^F}$ for regular $\theta$. Hence under the stabiliser condition, $|H_c ^*(X_0)_{\theta}|$ is isomorphic to some extension of ${\rho}_{\theta}$ by Lemma~\ref{lemm:existence of Heisenberg lift} (and again Lemma~\ref{lemm: statement (1)}).

\vspace{2mm} (ii) For $s\in T_1^F$, from the proof of Corollary~\ref{coro:character formula at rss} we know that
\begin{equation}\label{temp_formula_middle_sign}
\mathrm{Tr}(s,H^*_c(X_0)_{\theta})
=\theta(s)\cdot\frac{1}{|T_1^F|} \dim H_c^*(X_0\cap C_{TG^{l'}}(s))_{\theta|_{(T^1)^F}}.
\end{equation}
In particular, taking $s=1$ we see that 
\begin{equation}\label{temp formula: dimension 1}
\dim H^*_c(X_0)_{\theta}=\frac{1}{|T_1^F|}\dim H_c^*(X_0)_{\theta|_{(T^1)^F}}.
\end{equation}
Applying \eqref{temp formula: dimension 1} to $C_{G}(s)^{\circ}$ (instead of $G$) we get that
\begin{equation}\label{temp formula: dimension 2}
\dim H^*_c(X_0\cap C_{G}(s)^{\circ})_{\theta}=\frac{1}{|T_1^F|}\dim H_c^*(X_0\cap C_{G}(s)^{\circ})_{\theta|_{(T^1)^F}}.
\end{equation}
Since $X_0\subseteq TG^{l'}$ and $C_{TG^{l'}}(s)\subseteq C_G(s)^{\circ}$ (see e.g.\ Remark~\ref{remark: C_G(s)}), we have $X_0\cap C_{G}(s)^{\circ}=X_0\cap C_{TG^{l'}}(s)$, thus \eqref{temp_formula_middle_sign} and \eqref{temp formula: dimension 2} imply that
\begin{equation}\label{temp_formula_middle_sign final}
\mathrm{Tr}(s,H^*_c(X_0)_{\theta})=\theta(s)\cdot \dim H_c^*(X_0\cap C_{G}(s)^{\circ})_{\theta}.
\end{equation}
By (i) and the dimension property in Lemma~\ref{lemm:existence of Heisenberg lift} (again applying to $C_G(s)^{\circ}$, instead of $G$), we have
$$\dim \left| H_c^*(X_0\cap C_{G}(s)^{\circ})_{\theta} \right|=q^{\#\Phi_s^+},$$
so \eqref{temp_formula_middle_sign final} can be reformulated as:
\begin{equation}\label{formula:middle step}
\mathrm{Tr}(s,H^*_c(X_0)_{\theta})=e_{\theta}(s)\cdot q^{\#\Phi_s^+}\theta(s)
\end{equation}
for the function $e_{\theta}\colon T_1^F\rightarrow \{ 1,-1 \}$ given by
$$s\longmapsto\mathrm{sgn}(\dim H_c^*(X_0\cap C_{G}(s)^{\circ})_{\theta}).$$
So the assertion follows from basic properties of the Steinberg character.

\vspace{2mm} (iii) From the above argument we have $e_{\theta}(s)=\mathrm{sgn}(\dim H_c^*(X_0\cap C_{G}(s)^{\circ})_{\theta})$, so it suffices to show that 
$$\mathrm{sgn}(\dim H_c^*(X_0)_{\theta})=(-1)^{\mathrm{rk}_q({T_1})+\mathrm{rk}_q({G_1})}$$
for all the triples $(G,T,\theta)$ with strongly generic $\theta$, for $q\geq7$. For this we use an induction argument on $N:=\dim G_1$.

\vspace{2mm} If $N=1$, then the assertion is clear as $G_1=T_1$ is a torus. Suppose now that the assertion holds for all $N<m$ where $m\in \mathbb{Z}_{>1}$. Let $(G,T,\theta)$ be such that $N=m$.

\vspace{2mm} Consider the virtual character (of $T_1^F$):
$$\delta_{\theta}:=\mathrm{Tr}\left(-,(-1)^{\mathrm{rk}_q({T_1})+\mathrm{rk}_q({G_1})}H_c^*(X_0)_{\theta}|_{T_1^F}\right)-\mathrm{St}_{G_1}|_{T_1^F}\cdot\theta|_{T_1^F}.$$
Then it suffices to show that $\delta_{\theta}$ is zero for $q\geq7$. For any $s\in T_1^F$ not in the centre of $G_1$, the dimension of $C_{G_1}(s)^{\circ}$ is smaller than $m$, so by the induction assumption and \eqref{formula:middle step},  we have
\begin{equation*}
\delta_{\theta}(s)=0.
\end{equation*}
Meanwhile, if $s$ is in the centre of $G_1$, then (again by \eqref{formula:middle step}) we have
$$\delta_{\theta}(s)=((-1)^{\mathrm{rk}_q{(T_1)}+\mathrm{rk}_q{(G_1)}}e_{\theta}(1)-1)q^{\#\Phi^+}\theta(s).$$
Thus, for any $\xi\in\mathrm{Irr}(T_1^F)$ we see that
\begin{equation*}
\begin{split}
\langle\delta_{\theta},\xi\rangle_{T_1^F}
&=((-1)^{\mathrm{rk}_q{(T_1)}+\mathrm{rk}_q{(G_1)}}e_{\theta}(1)-1)\cdot\frac{q^{\#\Phi^+}}{|T_1^F|}\sum_{s\in Z(G_1)^F}\theta(s)\xi(s^{-1})\\
&=q^{\#\Phi^+}\cdot\frac{((-1)^{\mathrm{rk}_q{(T_1)}+\mathrm{rk}_q{(G_1)}}e_{\theta}(1)-1)|Z(G_1)^F|}{|T_1^F|}\cdot \langle\xi,\theta\rangle_{Z(G_1)^F}.
\end{split}
\end{equation*}
Note that this value needs to be an integer; however, as ${q^{\#\Phi^+}}$ and ${|T_1^F|}$ are coprime, and as $\langle\xi,\theta\rangle_{Z(G_1)^F}\in\{0,1\}$, this value can be a \emph{non-zero} integer only if
\begin{itemize}
\item[(B)] $((-1)^{\mathrm{rk}_q{(T_1)}+\mathrm{rk}_q{(G_1)}}e_{\theta}(1)-1)|Z(G_1)^F|$ is a non-zero integer divisible by $|T_1^F|$.
\end{itemize}
We shall show that (B) is not going to happen for $q\geq7$ (so that $\delta_{\theta}=0$).

\vspace{2mm} Note that, if $G_1$ is of semisimple rank zero, then $(-1)^{\mathrm{rk}_q{(T_1)}+\mathrm{rk}_q{(G_1)}}e_{\theta}(1)-1=0$ by construction. So the case~(B) may happen only if $G_1$ is of semisimple rank at least one with 
\begin{equation*}
|T_1^F|=|Z(G_1)^F|\quad \textrm{or} \quad 2|Z(G_1)^F|;   
\end{equation*}
it remains to show that in this case we actually have: If $q\geq7$ then  $|T_1^F/Z(G_1)^F|>2$.

\vspace{2mm} First, note that the natural injection $Z(G_1)^F\rightarrow Z(G_1)$ induces an injection of finite groups $Z(G_1)^F/(Z(G_1)^{\circ})^F\rightarrow Z(G_1)/Z(G_1)^{\circ}$, so
$$|T_1^F/Z(G_1)^F|
=\frac{|T_1^F/({Z(G_1)}^{\circ})^F|}{|Z(G_1)^F/({Z(G_1)}^{\circ})^F|}
\geq \frac{|T_1^F/({Z(G_1)}^{\circ})^F|}{|Z(G_1)/{Z(G_1)}^{\circ}|}.$$
Meanwhile, for the semisimple group $G_{\mathrm{ss}}:=G_1/{Z(G_1)}^{\circ}$ we have (see e.g.\ the argument of \cite[2.3.11]{Geck_Malle_2020book})
$$|T_1^F/({Z(G_1)}^{\circ})^F|=|(T_1/{Z(G_1)}^{\circ})^F|\geq (q-1)^{\mathrm{rk}(G_{\mathrm{ss}})},$$
where $\mathrm{rk}(-)$ denotes the rank, so
\begin{equation}\label{temp_formula:semisimple quot}
|T_1^F/Z(G_1)^F|
\geq \frac{(q-1)^{\mathrm{rk}(G_{\mathrm{ss}})}}{|Z(G_1)/{Z(G_1)}^{\circ}|}
\geq \frac{(q-1)^{\mathrm{rk}(G_{\mathrm{ss}})}}{|Z(G_{\mathrm{ss}})|}.
\end{equation}
Let $G_{\mathrm{sc}}$ be the simply-connected cover of $G_{\mathrm{ss}}$. Since $\overline{\mathbb{F}}_q^{\times}$ is divisible  (hence an injective object in the category of abelian groups), by \cite[1.5.2 and 1.5.3]{Geck_Malle_2020book} the group $Z(G_{\mathrm{ss}})$ is a quotient of $Z(G_{\mathrm{sc}})$, thus from \eqref{temp_formula:semisimple quot} we get that
$$|T_1^F/Z(G_1)^F|
\geq \frac{(q-1)^{\mathrm{rk}(G_{\mathrm{ss}})}}{|Z(G_{\mathrm{sc}})|}
=\frac{(q-1)^{\mathrm{rk}(G_{\mathrm{sc}})}}{|Z(G_{\mathrm{sc}})|}.$$
As a simply-connected group, $G_{\mathrm{sc}}$ is the direct product of its simple factors $S_i$ (see e.g.\ \cite[1.5.10]{Geck_Malle_2020book}), so it suffices to show that, if $q\geq 7$ then 
$$\frac{(q-1)^{\mathrm{rk}(S_i)}}{|Z(S_i)|}>2.$$
The above statement follows from the explicit descriptions of centres of simply-connected algebraic groups listed in \cite[1.5.6]{Geck_Malle_2020book}: If $S_i$ is not of type $\mathsf{A}_n$ (in particular, $\mathrm{rk}(S_i)\geq2$), then $|Z(S_i)|\leq 4$, so any $q\geq 4$ makes the inequality hold; if  $S_i$ is of type $\mathsf{A}_n$, then $|Z(S_i)|\leq n+1$, so any $q\geq7$ makes the inequality hold.
\end{proof}

For $r>1$ and $\theta$ strongly generic, let $e_{\theta}\colon T_1^F\rightarrow\{\pm1\}$ be as in Lemma~\ref{lemm:full alg of H(X_0)} if $r$ is odd, and $=1$ if $r$ is even. Now we can present the algebraisation of $R_{T,U}^{\theta}$ for an arbitrary integer $r>1$ (both even and odd):

\begin{thm}\label{thm:full alg}
Let $r>1$. Suppose that $\theta$ is strongly generic. Then 
$$R_{T,U}^{\theta}\cong e_{\theta}(1)^r\cdot \mathrm{Ind}_{(TG^{l'})^F}^{G^F} \hat{\rho}_{\theta},$$
where $\hat{\rho}_{\theta}\in\mathrm{Irr}((TG^{l'})^F)$ is uniquely characterised by the following two properties:
\begin{itemize}
\item[(a)] If $s\in (T^1G^{l'})^F$, then
$$\mathrm{Tr}(s,\hat{\rho}_{\theta})=\left(\frac{1}{q^{\#\Phi^+}}\right)^{l-l'}\cdot\mathrm{Tr}\left(s,\mathrm{Ind}_{(T^1G^l)^F}^{(T^1G^{l'})^F}\widetilde{\theta}|_{(T^1G^l)^F}\right);$$
\item[(b)] if $s\in T_1^F$, then
$$\mathrm{Tr}(s,\hat{\rho}_{\theta})=(-1)^{\mathrm{rk}_q{(G_1)}+\mathrm{rk}_q({C_{G_1}(s)^{\circ}})}\cdot\left(e_{\theta}(1)e_{\theta}(s)\right)^r\cdot  \mathrm{St}_{G_1}(s)^{l-l'}\cdot\theta(s).$$
\end{itemize}
If we assume further that $q\geq7$, then $e_{\theta}(s)=(-1)^{\mathrm{rk}_q{(T_1)}+\mathrm{rk}_q({C_{G_1}(s)^{\circ}})}$ when $r$ is odd, so all the terms in the above are explicitly determined for $q\geq7$.
\end{thm}
\begin{proof}
This follows immediately from Theorem~\ref{thm:main_even}, Theorem~\ref{thm:induction theorem}, and Lemma~\ref{lemm:full alg of H(X_0)}.
\end{proof}

\begin{remark}\label{remark: on the sign function}
The above theorem implies that, for every $r>1$, if $\theta$ is strongly generic and $q\geq 7$, then $(-1)^{r(\mathrm{rk}_q{(G_1)}+\mathrm{rk}_q{(T_1)})}R_{T,U}^{\theta}$ is isomorphic to G\'erardin's representation, whenever the latter is defined (G\'erardin imposed further restrictions on $(G,T,\theta)$; see \cite[4.3.4]{Gerardin1975SeriesDiscretes}). This gives a positive solution to the problem raised by Lusztig in \cite[Introduction]{Lusztig2004RepsFinRings}.
\end{remark}

\begin{remark}\label{remark:q<7 for GL_n}
For $\mathrm{GL}_n$, one can remove the assumption ``$q\geq7$'' in both Lemma~\ref{lemm:full alg of H(X_0)} and Theorem~\ref{thm:full alg}: Indeed, for any non-split $T_1$, one easily checks that $|T_1^F|/|Z(G_1)^F|>2$ for every $q$, so (B) in the proof of Lemma~\ref{lemm:full alg of H(X_0)}~(iii) does not happen for any $q$.
\end{remark}

\begin{coro}\label{coro:dim formula for general case}
Let $r>1$. Suppose that $\theta\in\mathrm{Irr}({T^F})$ is strongly generic. Then 
\begin{equation*}
\dim |R_{T,U}^{\theta}|={|G_1^F/T_1^F|_{p'}}\cdot q^{(r-1)\cdot\#\Phi^+}.
\end{equation*}
\end{coro}
\begin{proof}
This follows immediately from Theorem~\ref{thm:full alg}.
\end{proof}
Note that, the expression of this dimension formula is true also for $r=1$, in which case the proof uses the Steinberg character; see \cite[12.9]{DM1991}. On the other hand, for $r=1$ one can remove the conditions on $\theta$, but for $r>1$ usually this cannot be done, as indicated by Lusztig's example \cite[3.4]{Lusztig2004RepsFinRings}.

\section{Proof of Lemma~\ref{lemm: statement (1)}}\label{sec: tech1}

First note that the map
$$(g,g')\longmapsto (L(g),L(g'),g^{-1}g)$$
gives a $T^F\times T^F$-equivariant isomorphism
$$(TG^{l'})^F\backslash(X_0\times X_0)\longrightarrow \Sigma:=\{ (x,x',y)\in FU^{l',l}\times FU^{l',l} \times TG^{l'}) \mid xF(y)=yx' \}.$$
Here $T^F\times T^F$ acts on $\Sigma$ via $(t,t')\colon (x,x',y)\mapsto (x^t,x'^{t'},t^{-1}yt')$. So by the K\"unneth formula we see that
$$\langle H_c ^*(X_0)_{\theta},H_c ^*(X_0)_{\theta}\rangle_{(TG^{l'})^F}=\sum_i(-1)^i\dim H_c^i(\Sigma,\overline{\mathbb{Q}}_{\ell})_{\theta^{-1},\theta}.  $$

\vspace{2mm} Now, by the Iwahori decomposition on $G^{l'}$ we can rewrite $\Sigma$ as
$$\{ (x,x',u^{-},\tau,u)\in  FU^{l',l}\times FU^{l',l} \times (U^-)^{l'}\times T\times U^{l'} \mid xF(u^-\tau u)=u^-\tau ux'    \},$$
on which $T^F\times T^F$ acts by
$$(t,t')\colon (x,x',u^{-},\tau,u)\longmapsto (x^t,x'^{t'},(u^{-})^t,t^{-1}\tau t',u^{t'}).$$
Furthermore, the variable change $x'F(u^{-1})\mapsto x'$ identifies $\Sigma$ with
$$\widetilde{\Sigma}:=\{ (x,x',u^{-},\tau,u)\in  FU^{l',l}\times FU^{l',l} \times (U^-)^{l'}\times T\times U^{l'} \mid xF(u^-\tau )=u^-\tau ux'    \}$$
on which the $T^F\times T^F$-action does not change.

\vspace{2mm} Now consider the partition into a closed subvariety and an open subvariety
$$\widetilde{\Sigma}:=\Sigma_c \sqcup \Sigma_o, $$
where in $\Sigma_c$ the component $u^-$ lies in $(U^-)^l$, and in $\Sigma_o$ the component $u^-$ lies in $(U^-)^{l'}\backslash (U^-)^l$. Clearly, both $\Sigma_c$ and $\Sigma_o$ are stable under the $T^F\times T^F$-action, so the assertion follows from the following two lemmas:

\begin{lemm}\label{lemm:lemm1 for statement (1)}
We have $\dim H_c^*(\Sigma_c)_{\theta^{-1},\theta}=1$.
\end{lemm}

\begin{lemm}\label{lemm:lemm2 for statement (1)}
If $\theta$ is regular, then $\dim H_c^j(\Sigma_o)_{\theta^{-1},\theta}=0$ for all $j$.
\end{lemm}

\begin{proof}[Proof of Lemma~\ref{lemm:lemm1 for statement (1)}]
The change of variables $xF(u^{-})\mapsto x$ identifies $\Sigma_c$ with
$$\widetilde{\Sigma}_c:=\{ (x,x',u^{-},\tau,u)\in  FU^{l',l}\times FU^{l',l} \times (U^-)^{l}\times T\times U^{l'} \mid xF(\tau)=u^-\tau ux'    \}$$
($T^F\times T^F$-equivariantly). On $\widetilde{\Sigma}_c$, the $T^F\times T^F$-action naturally extends to the action of 
$$H:=\{ (t,t)\mid t\in T_1  \}$$
(note that $T_1$ is always a subgroup of $T$). As $H$ is a torus, by basic properties of $\ell$-adic cohomology we have
$$\dim H_c^*(\Sigma_c)_{\theta^{-1},\theta}=\dim H_c^*((\widetilde{\Sigma}_c)^{H})_{\theta^{-1},\theta}.$$
On the other hand, by construction it is clear that
$$(\widetilde{\Sigma}_c)^{H}= \{ (1,1,1,\tau,1) \mid \tau\in T^F \} \cong T^F.$$
So the space $H_c^*(\Sigma_c)_{\theta^{-1},\theta}$ can be identified with the $1$-dimensional space $\theta^{-1}\times \theta$, which completes the proof.
\end{proof}

\begin{proof}[Proof of Lemma~\ref{lemm:lemm2 for statement (1)}]
First note that it suffices to show
\begin{equation}\label{formula:Sigma_o}
H_c^j(\Sigma_o)_{\theta^{-1}|_{(T^{r-1})^F}}=0,
\end{equation}
where $(T^{r-1})^F$ acts on $\Sigma_o$ by $t\colon (x,x',u^-,\tau,u)\mapsto (x,x',u^-,t^{-1}\tau,u)$.

\vspace{2mm} To continue the argument we shall need to stratify $\Sigma_o$ according to roots. For $\beta\in\Phi^-$, let $F(\beta)$ be the root such that $F(U_{\beta})=U_{F(\beta)}$, and let us fix a total order on $F(\Phi^-)$ refining the height function $\mathrm{ht}(-)$. Then every element $u'$ in $F((U^-)^{l'}\backslash (U^-)^l)$ can be written uniquely as $\prod_{\beta\in\Phi^{-}}x_{F(\beta)}^{u'}$, where $x_{F(\beta)}^{u'}\in U_{F(\beta)}$ and the order is taken to be: If $F(\beta_1)<F(\beta_2)$, then $x_{F(\beta_1)}^{u'}$ is at the left to $x_{F(\beta_2)}^{u'}$. Then we have a stratification
$$F((U^-)^{l'}\backslash (U^-)^l)=\bigsqcup_{\beta\in\Phi^-}U_{F(\beta)+},$$
where $U_{F(\beta)+}$ is the locally closed subvariety (of $F((U^-)^{l'}\backslash (U^-)^l)$) consisting of the elements $u'$ such that
\begin{itemize}
\item [(i)] $x_{F(\beta)}^{u'}\notin F((U^-)^l)$, and
\item [(ii)] $x_{F(\beta_1)}^{u'}\in F((U^-)^l)$ for all $F(\beta_1)<F(\beta)$;
\end{itemize}
in particular, for $u'\in U_{F(\beta)+}$ we have 
\begin{equation}\label{formula:place change expression}
u'=x_{F(\beta)}^{u'}\cdot \prod_{\beta'\neq\beta}x_{F(\beta')}^{u'},
\end{equation}
where the product order of $\prod_{\beta'\neq\beta}x_{F(\beta')}^{u'}$ is as before. Note that the above stratification gives a $(T^{r-1})^F$-equivariant stratification of $\Sigma_o$:
$$\Sigma_o=\coprod_{\beta\in\Phi^-}\Sigma_{\beta},$$
where $\Sigma_{\beta}$ is the locally closed subvariety characterised by the property that the component $u^-$ satisfying $F(u^-)\in  U_{F(\beta)+}$. Therefore, to prove \eqref{formula:Sigma_o} it suffices to show that
\begin{equation}\label{formula:Sigma_beta}
H_c^j(\Sigma_{\beta})_{\theta^{-1}|_{(T^{r-1})^F}}=0
\end{equation}
for each $\beta\in\Phi^{-}$.

\vspace{2mm} To proceed we need a technical commutator result, which has been repeatedly used in \cite{Lusztig2004RepsFinRings}, \cite{Sta2009Unramified}, \cite{ChenStasinski_2016_algebraisation}, \cite{Chen_2017_InnerProduct}, in slightly different forms; however, in the current situation a much simpler version is enough for us, as we are working inside $TG^{l'}$ instead of $G$:

\vspace{2mm} Consider the group 
$$H_{\beta}:=\{ t\in T^{r-1}  \mid L(t)\in \mathcal{T}^{-F(\beta)}  \},$$
where $\mathcal{T}^{-F(\beta)}:=(T^{-F(\beta)})^{r-1} $ is the smallest reduction kernel of the root subgroup, as defined in Section~\ref{sec:prelim}. Note that in particular $H_{\beta}$ contains $(T^{r-1})^F$. 

\begin{lemm}
There is a morphism 
$$\Psi\colon H_{\beta}\times U_{F(\beta)+}\longrightarrow FU^{l'},$$
such that the $T$-part (in the Iwahori decomposition $G^{r-1}=T^{r-1}U^{r-1}(U^-)^{r-1}$) of the commutator $[u',\Psi(t,u')]\in G^{r-1}$ is $F(t^{-1})t=L(t)^{-1}$, and that $\Psi(t,-)$ is the constant map with value $1$ when $F(t)=t$.
\end{lemm}

\begin{proof}
Following \cite[XX]{SGA3}, we fix an isomorphism $p_{\alpha}\colon (\mathbb{G}_a)_{\mathcal{O}_r^{\mathrm{ur}}}\cong \mathbf{U}_{\alpha}$ for every $\alpha$, in a way that there exists $a\in \mathbb{G}_m(\mathcal{O}_r^{\mathrm{ur}})$ such that, for all $x,y\in \mathbb{G}_a(\mathcal{O}_r^{\mathrm{ur}})$ we have
\begin{equation}\label{formula:opposite roots}
p_{-\alpha}(y)p_{\alpha}(x)=p_{\alpha}\left(\frac{x}{1+axy}\right)\check{\alpha}\left(\frac{1}{1+axy}\right)p_{-\alpha}\left(\frac{y}{1+axy}\right)
\end{equation}
(see \cite[XX~2.2]{SGA3}).

\vspace{2mm} Now for $u'\in U_{F(\beta)+}$, we write $u'=x^{u'}_{F(\beta)}\cdot \prod_{\beta'\neq\beta}x_{F(\beta')}^{u'}$ as in \eqref{formula:place change expression}, and take $u'':=\prod_{\beta'\neq\beta}x_{F(\beta')}^{u'}$. Let us put 
$$\lambda(t):=\left(1-\left( \check{-F(\beta)} \right)^{-1}(F(t^{-1})t)\right)\cdot\frac{1}{a}\cdot\frac{1}{p^{-1}_{F(\beta)}(x^{u'}_{F(\beta)})^2}\in\overline{\mathbb{F}}_q$$ 
and define
$$\Psi(t,u'):=p_{-F(\beta)}\left(\lambda(t)\cdot p_{F(\beta)}^{-1}(x^{u'}_{F(\beta)})\right).$$
We shall verify that this construction satisfies the desired property. Note that
$$[u',\Psi(t,u')]={^{x^{u'}_{F(\beta)}}[u'',\Psi(t,u')]}\cdot [x^{u'}_{F(\beta)},\Psi(t,u')].$$
By inductively applying the Chevalley commutator formula \cite[3.3.4.1]{Demazure_Summary_of_Thesis} one sees that ${^{x^{u'}_{F(\beta)}}[u'',\Psi(t,u')]}\in U^{r-1}(U^-)^{r-1}$ (see e.g.\ the argument of \cite[Lemma~4.6]{ChenStasinski_2016_algebraisation}), so it does not contribute to the $T$-part in the Iwahori decomposition. For the second part $[x^{u'}_{F(\beta)},\Psi(t,u')]$, by \eqref{formula:opposite roots} we have (note that $p^{-1}_{F(\beta)}(x^{u'}_{F(\beta)})^3=0$)
\begin{equation*}
\begin{split}
[x^{u'}_{F(\beta)},\Psi(t,u')] ={}
& p_{F(\beta)}(p^{-1}_{F(\beta)}(x^{u'}_{F(\beta)}))\\
& \times p_{-F(\beta)}\left(\lambda(t)\cdot p_{F(\beta)}^{-1}(x^{u'}_{F(\beta)})\right) \times p_{F(\beta)}(-p^{-1}_{F(\beta)}(x^{u'}_{F(\beta)}))\\
& \times p_{-F(\beta)}\left(-\lambda(t)\cdot p_{F(\beta)}^{-1}(x^{u'}_{F(\beta)})\right)\\
={} & p_{F(\beta)}(p^{-1}_{F(\beta)}(x^{u'}_{F(\beta)}))\times p_{F(\beta)}(-p^{-1}_{F(\beta)}(x^{u'}_{F(\beta)}))\\
& \times \check{F(\beta)}(1+a \cdot \lambda(t) p^{-1}_{F(\beta)}(x^{u'}_{F(\beta)})^2 ) \\
& \times p_{-F(\beta)}\left(\lambda(t)\cdot p_{F(\beta)}^{-1}(x^{u'}_{F(\beta)})\right) \times p_{-F(\beta)}\left(-\lambda(t)\cdot p_{F(\beta)}^{-1}(x^{u'}_{F(\beta)})\right)\\
={} & \check{(-F(\beta))}(1-a\lambda(t) p^{-1}_{F(\beta)}(x^{u'}_{F(\beta)})^2 )\\
={} & \check{(-F(\beta))}\left( \left( \check{-F(\beta)} \right)^{-1}(F(t^{-1})t)  \right)=F(t^{-1})t,
\end{split}
\end{equation*}
as desired. On the other hand, by construction $\Psi(t,-)$ is constant when $F(t)=t$. This completes the proof.
\end{proof}

Now we return to the proof of Lemma~\ref{lemm:lemm2 for statement (1)}. First, by the variable change $\tau^{-1} u^-\tau\mapsto u^-$ we can rewrite $\Sigma_{\beta}$ as
$$\widetilde{\Sigma}_{\beta}=\{ (x,x',u^{-},\tau,u)\in  FU^{l',l}\times FU^{l',l} \times U_{F(\beta)+}\times T\times U^{l'} \mid xF(\tau u^- )= \tau  u^- ux'    \},$$
on which the $(T^{r-1})^F$-action does not change, and to show \eqref{formula:Sigma_beta} is equivalent to showing
\begin{equation}\label{formula:Sigma_beta2}
H_c^j(\widetilde{\Sigma}_{\beta})_{\theta^{-1}|_{(T^{r-1})^F}}=0.
\end{equation}
By construction, one verifies immediately that the morphism $\Psi$ gives a morphism
$$h\colon H_{\beta}\times \widetilde{\Sigma}_{\beta}\longrightarrow \widetilde{\Sigma}_{\beta};\quad (t,x,x',u^-,\tau,u)\longmapsto (f_{t}(x),g_t(x'),u^-,t^{-1}\tau,u),$$
where
$$g_t\colon x'\longmapsto x'\cdot \Psi(t,u^-)$$
and
$$f_t\colon x\longmapsto x\cdot {^{F(\tau)}  \left(  t^{-1}F(t) \cdot  [F(u^-),x'^{-1}g_t(x')]\cdot x'^{-1}g_t(x') \right)  } = x\cdot {^{F(\tau u^-)}  \left(  t^{-1}F(t) \cdot   \Psi(t,u^-)   \right)  }.$$
Note that, by construction, $h(t,-)$ extends the $(T^{r-1})^F$-action on $\widetilde{\Sigma}_{\beta}$, and it is an automorphism of $\widetilde{\Sigma}_{\beta}$ for every $t\in H_{\beta}$. Moreover, by the homotopy property of cohomology (see e.g.\ \cite[Page~136]{DL1976} or \cite[4.4]{ChenStasinski_2016_algebraisation}) the induced map of $h(t,-)$ on $H_c^j(\widetilde{\Sigma}_{\beta})$ is the identity map for all $t\in (H_{\beta})^{\circ}$.

\vspace{2mm} In such a situation it follows from a routine argument that the regularity of $\theta$ and the construction of $H_{\beta}$ imply the property \eqref{formula:Sigma_beta2}: Let $a\in\mathbb{Z}_{>0}$ be such that $F^a(\mathcal{T}^{-F(\beta)})=\mathcal{T}^{-F(\beta)}$. By continuity, the images of the norm map $N^{F^a}_F(t)=t\cdot F(t)\cdots F^{a-1}(t)$ on $\mathcal{T}^{-F(\beta)}$ form a connected subgroup of $H_{\beta}$, hence are contained in $(H_{\beta})^{\circ}$. Therefore $N^{F^a}_F(\mathcal{T}^{-F(\beta)})^{F^a})\subseteq (T^{r-1})^F\cap (H_{\beta})^{\circ}$, hence it must act on $H^j_c(\widetilde{\Sigma})$ trivially. Thus the regularity of $\theta$ implies that
\begin{equation*}
H^j_c(\widetilde{\Sigma}_{\beta})_{\theta^{-1}\big|_{ N^{F^a}_F\left(\left(\mathcal{T}^{-F(\beta)}\right)^{F^a}\right)}}=0.
\end{equation*}
In particular, $H^j_c(\widetilde{\Sigma}_{\beta})_{\theta^{-1}|_{(T^{r-1})^F}}=0$. This completes the proof.
\end{proof}

\bibliographystyle{alpha}
\bibliography{zchenrefs}

\end{document}